\documentclass[smallextended]{svjour3}%
\usepackage{textcomp}
\usepackage{times}
\usepackage{hyperref}
\usepackage{amsmath}
\usepackage{amsfonts}
\usepackage{amssymb}
\usepackage{graphicx}%
\numberwithin{equation}{section}

\begin{document}
\title{Diameter bounded equal measure partitions of Ahlfors regular metric measure spaces}
\author{Giacomo Gigante \and Paul Leopardi}
\institute{Giacomo Gigante \at University of Bergamo \email{giacomo.gigante@unibg.it} \and
           Paul Leopardi \at University of Newcastle, Australia \email{paul.leopardi@gmail.com}} 
\journalname{Discrete and Computational Geometry}
\maketitle

\begin{abstract}
The algorithm devised by Feige and Schechtman for partitioning higher
dimensional spheres into regions of equal measure and small diameter is
combined with David and Christ's construction of dyadic cubes to yield a partition
algorithm suitable to any connected Ahlfors regular metric measure space of
finite measure.
\begin{subclass}
~Primary 52C22; \and Secondary 11K38, \and 28A75, \and 54E45, \and 65D30.
\end{subclass}

\end{abstract}


\authorrunning{G. GIGANTE, P. LEOPARDI}
\titlerunning{EQUAL AREA PARTITIONS OF METRIC MEASURE SPACES}

\section{Introduction}

Stolarsky \cite[p.~581]{Sto73} asserts the existence for any natural number
$N$ of a partition of the unit sphere $\mathbb{S}^{d}\subset\mathbb{R}^{d+1}$
into $N$ regions of equal volume and small diameter, { that is, diameter smaller than $\kappa N^{-1/d} $
for some fixed constant $\kappa$. 

A practical means to achieve such a partition is the recursive zonal equal
volume (EQ) sphere partitioning algorithm \cite[Section 3]{Leo07}, which extends to any dimension 
the algorithm studied in \cite{Rak94,Zho95} (see also \cite{Saf97}) for the case of the two-dimensional sphere.
In these works, particular attention is devoted to obtaining a constant $\kappa$ as small as possible.}

Feige and Schechtman
\cite{FeiS02} give a constructive proof of the following lemma, which can also
be used to prove Stolarsky's assertion.

\begin{lemma}
\label{Feige-Schechtman-original-lemma} \cite[Lemma~21, pp.~430--431]{FeiS02}
For each $0<\gamma<\pi/2$ the sphere $\mathbb{S}^{d-1}$ can be partitioned
into $N=\big(\operatorname{O}(1)/\gamma\big)^{d}$ regions of equal volume,
each of diameter at most $\gamma$.
\end{lemma}

The Feige-Schechtman construction can be modified to give an equal volume
partition of the sphere $\mathbb{S}^{d}$ with a diameter bound of order
$\operatorname{O}(N^{-1/d})$ \cite{Leo09}. It can
also be modified to give a construction which yields an equal measure
partition of a compact connected Riemannian manifold, as suggested in \cite{Leo13}. 
The key observation is
that for such a manifold $X$ with measure $\mu(X)<\infty$, the measure of
metric balls is bounded both from below and from above. More precisely, there
are positive global constants $c_{1}<c_{2}$ and $d$ such that all open metric
balls $B(x,r)$ with $x\in X$, $0<r\leqslant\operatorname{diam}(X)$ satisfy the
bounds
\[
c_{1}r^{d}\leqslant\mu\big(B(x,r)\big)\leqslant c_{2}r^{d}.
\]
In other words, a compact connected Riemannian manifold is a special case of
an Ahlfors regular metric measure space \cite[Chapter 1]{DavS97}. Therefore,
given a positive integer $N$ sufficiently large, if $r:=(c_{1}N)^{-1/d}$ then
all balls of radius $r$ will have measure at least $N^{-1}$. The construction
on a compact manifold $X$ of finite measure therefore starts with a saturated
packing of balls of radius $r$ and then proceeds as per the construction on
the sphere. We omit the details, since our main result includes the case of 
compact connected Riemannian manifolds.

In this paper, we describe yet another modified construction that gives a
diameter bounded equal measure partition of a connected Ahlfors regular metric
measure space with a finite measure. The construction also yields a partition
containing a well separated set of points, that is each region contains a ball with
radius comparable to the diameter of the region. The main modification in the
construction is the use of David and Christ's dyadic cubes \cite{Chr90,Dav88,Dav91,DavS97} 
in place of Voronoi cells.

{ Small diameter equal measure partitions are often used to prove results related to discrepancy.
Recently, M. Skriganov \cite{Skr15} has given a constructive proof of the existence of equal measure 
partitions with small average diameter for compact $d-$rectifiable metric spaces. He has then used this 
construction to prove a generalization of Stolarsky's invariance principle. 

This type of partition is also essential in the proof of the existence of point distributions with 
small discrepancy with respect to properly chosen collections of sets 
(think, for example, of the discrepancy with respect to metric balls in compact manifolds). 
See \cite[Theorems 24A and 24D, pages 181--182]{BC} for the case of spherical cap discrepancy on the sphere, 
or \cite{BCCGT} for more general sets and metric spaces.

Some applications to numerical integration on metric spaces are described in Section \ref{numerical}.}

\section{Dyadic cubes on an Ahlfors regular space}

\begin{definition} (\cite[Chapter 1]{DavS97}, \cite[page 413]{Gro})
\label{Ahlfors-def} An Ahlfors regular metric measure space of dimension $d>0$
is a complete metric space $X$ with a Borel measure $\mu$ with the property
that there are two positive constants $c_{1}$ and $c_{2}$ such that all open
metric balls $B(x,r)$ with $x\in X$, $0<r\leqslant\operatorname{diam}(X)$
satisfy the bounds
\begin{equation}
c_{1}r^{d}\leqslant\mu\left(  B(x,r)\right)  \leqslant c_{2}r^{d}.
\label{AR-ball-bounds}%
\end{equation}

\end{definition}


\begin{definition}
\label{DC-dyadic-cube-props-def} A collection of open subsets of $X$,
$\left\{  Q_{\alpha}^{k}\subset X:k\in\mathbb{Z},\,\alpha\in I_{k}\right\}  $
is a family of dyadic cubes of $X$ if there exist three constants $\delta
\in\left(  0,1\right)  $ and $0<a_{0}\leqslant a_{1},\,$such that the following
properties hold
\begin{align}
&  \mu\left(  X\setminus\bigcup_{\alpha\in I_{k}}Q_{\alpha}^{k}\right)
=0\quad\text{for all~}k.\label{DC-prop-1}\\
&  Q_{\alpha}^{k}\cap Q_{\beta}^{k}=\emptyset\quad\text{for each }k\text{ and
}\alpha\neq\beta.\label{DC-prop-2}\\
&  \text{If~}\ell>k\text{~then either~}Q_{\beta}^{\ell}\subset Q_{\alpha}%
^{k}\text{~or~}Q_{\beta}^{\ell}\cap Q_{\alpha}^{k}=\emptyset.
\label{DC-prop-3}\\
&  \text{Each~}Q_{\alpha}^{k}\text{~contains a ball~}B(z_{\alpha}^{k}%
,a_{0}\delta^{k}).\label{DC-prop-inner-ball}\\
&  \text{Each~}Q_{\alpha}^{k}\text{~is contained in the ball~}B(z_{\alpha}%
^{k},a_{1}\delta^{k}). \label{DC-prop-outer-ball}%
\end{align}

\end{definition}

In \cite{Dav88,Dav91} (see also \cite{DavS97}), G. David shows that Ahlfors
regular metric measure spaces contained in Euclidean spaces admit a dyadic
cube decomposition. By means of Assouad embedding theorem \cite{Ass83} this
decomposition holds for general Ahlfors regular metric measure spaces. Later,
M. Christ \cite[Section 3]{Chr90} gave a direct construction of a dyadic cube
decomposition for the more general case of spaces of homogeneous type.

In fact, both David and Christ's dyadic cubes also have the property
\[
\mu(\{x\in Q_{\alpha}^{k}:\rho(x,X\setminus Q_{\alpha}^{k})\leqslant
t\delta^{k}\})\leqslant Ct^{\eta}\mu(Q_{\alpha}^{k}) \quad\forall
k,\alpha,\quad\forall t>0,
\]
which means that the mass of each cube is concentrated away from its boundary,
and its boundary has measure zero. We shall not use this property.

Finally, it is easy to see that by carefully attaching part of its boundary to each  $Q_\alpha^k,$  
property \eqref{DC-prop-1} can be replaced by the stronger property
\[
\bigcup_{\alpha\in I_{k}}Q_{\alpha}^{k}=X\quad\text{for all~}k,
\]
maintaining all other properties (of course the cubes will no longer be open in general). 

\begin{theorem}
\label{DC-thm-dyadic-cubes} (\cite[Theorem 11]{Chr90}, \cite[Lemma 11 ]{Dav88}, \cite[p. 86]{Dav91},
\cite[p. 22]{DavS97})(Dyadic cubes).
Let $X$ be an Ahlfors regular metric measure space of dimension $d.$ Then
there exists a family of dyadic cubes
as in Definition \ref{DC-dyadic-cube-props-def}.
\end{theorem}

The above construction immediately provides a partition of $X$ into $N$
regions of size approximately $N^{-1/d}$, and therefore of measure
\emph{approximately} $N^{-1}.$ Notice that $X$ is not required to be connected.

\begin{corollary}
\label{quasi-equal}
Let $X$ be an Ahlfors regular metric measure space of dimension $d$ and finite measure. Then
there exist positive constants $c_{3},c_{4}$ such that for every positive
integer $N$ there is a partition of $X$ into $N$ regions each contained in a
ball of radius $c_{3}N^{-1/d}$ and containing a ball of radius $c_{4}N^{-1/d}$.
\end{corollary}

\begin{proof}
For any integer $n$ such that $a_0\delta^n\leqslant\operatorname{diam}(X),$ let $N_{n}$ be the cardinality of the set $I_{n}$, that is
the number of cubes of generation $n$ within $X.$ By the Ahlfors regularity of $X$ and
the properties of the dyadic cube partition,
\begin{align*}
N_{n}c_{1}\left(  a_{0}\delta^{n}\right)  ^{d}
&\leqslant\sum_{\alpha\in I_{n}} \mu\left( B \left( z_{\alpha}^{n},a_{0}\delta^{n} \right) \right)
\leqslant\mu\left( X \right)  
\\
&\leqslant\sum_{\alpha\in I_{n}} \mu\left( B \left( z_{\alpha}^{n},a_{1}\delta^{n} \right) \right)  
\leqslant N_{n} c_{2} \left( a_{1}\delta^{n} \right) ^{d},
\end{align*}
so that
\[
\frac{\mu\left(  X\right)  }{c_{2}a_{1}^{d}\delta^{nd}}\leqslant N_{n}\leqslant\frac
{\mu\left(  X\right)  }{c_{1}a_{0}^{d}\delta^{nd}}.
\]
Let us now take the integer $n$ such that
\[
N_{n}\leqslant N<N_{n+1}.
\]
Observe that
\[
\frac{N}{N_{n}}<\frac{N_{n+1}}{N_{n}}\leqslant\frac{\mu\left(  X\right)  }%
{c_{1}a_{0}^{d}\delta^{nd+d}}\frac{c_{2}a_{1}^{d}\delta^{nd}}{\mu\left(
X\right)  }=\frac{c_{2}a_{1}^{d}}{c_{1}a_{0}^{d}\delta^{d}}=:H.
\]
Let now $k$ be any  integer such that
\[
Hc_2a_{1}^d\delta^{kd}\leqslant c_1a_{0}^d.
\]
Since every cube of generation $n+k$ has measure bounded above
by $c_2 a_1^d\delta^{(n+k)d},$ while every cube of generation $n$ has measure bounded 
below by $c_1 a_0^d\delta^{nd},$
it follows that every cube $Q_{\alpha}^{n}$ 
contains at least $H$ cubes of generation $n+k,$ each containing a ball of
radius
\[
a_{0}\delta^{n+k}\geqslant\frac{a_{0}\delta^{k}}{a_{1}}\left(  \frac{\mu\left(
X\right)  }{c_{2}N_{n}}\right)  ^{1/d}\geqslant\frac{a_{0}\delta^{k}}{a_{1}}\left(
\frac{\mu\left(  X\right)  }{c_{2}}\right)  ^{1/d}\frac{1}{N^{1/d}}.
\]
Furthermore, every cube $Q_{\alpha}^{n}$  is contained in a ball of radius
\[
a_{1}\delta^{n}=\frac{a_{1}}{\delta}\delta^{n+1}\leqslant\frac{a_{1}}{\delta a_{0}%
}\left(  \frac{\mu\left(  X\right)  }{c_{1}N_{n+1}}\right)  ^{1/d}\leqslant
\frac{a_{1}}{\delta a_{0}}\left(  \frac{\mu\left(  X\right)  }{c_{1}}\right)
^{1/d}\frac{1}{N^{1/d}}%
\]
By taking unions of cubes of generation $n+k$ contained in the same cube of
generation $n$, divide $X$ into exactly $N$ regions satisfying the required
properties. This can be done because there are $N_{n}\leqslant N$ cubes of
generation $n$ and at least $N_{n}H> N$ cubes of generation $n+k$.
\end{proof}

The construction of dyadic cubes as per Definition \ref{DC-dyadic-cube-props-def}
ensures that each cube of a fixed generation is contained in a ball of
a fixed radius \eqref{DC-prop-outer-ball}.
A similar property also holds for each cube and its immediate
neighbours in a fixed generation, in the following sense.
\begin{corollary}
\label{DC-big-ball-cor} For all $k$ and all $\beta\in I_{k},$ the union of cubes
\begin{equation}
\label{DC-union-of-cubes-def}
Q_{\beta}^{k}\cup\bigcup_{B(z_{\alpha}^{k},a_1\delta^{k})
\cap 
B(z_{\beta}^{k},a_1\delta^{k})\neq\emptyset}Q_{\alpha}^{k}%
\end{equation}
is contained in the ball $B(z_{\beta}^{k},3  a_{1}  \delta^{k})$.
\end{corollary}

\begin{proof}
By the triangle inequality, the distance between two centre points $z_{\alpha
}^{k},z_{\beta}^{k}$ of the pair of overlapping balls $B(z_{\alpha}^{k}%
,a_1\delta^{k}),B(z_{\beta}^{k},a_1\delta^{k})$ is at most $2a_1\delta^{k}$. By
property \eqref{DC-prop-outer-ball} of dyadic cubes, each cube $Q_{\alpha}^{k}$
is contained in the ball $B(z_{\alpha}^{k},a_{1}\delta^{k})$, and thus, by the
triangle inequality, the union of cubes \eqref{DC-union-of-cubes-def} is
contained in the ball $B(z_{\beta}^{k},3  a_{1}  \delta^{k})$.
\end{proof}

Ahlfors regularity of $X$ implies immediately that the measure $\mu$ is
non-atomic. By a result due to Sierpinski \cite{Sier22}, non atomic measures
take all values between $0$ and $\mu\left(  X\right)  .$ The proof in the
general case is based on Zorn's lemma or functional analysis techniques (see
\cite[p. 38]{Fry04}), but in our case it follows immediately from the above
dyadic decomposition.

\begin{corollary}
\label{DC-measurable-subset-corollary} Let $X$ be an Ahlfors regular
metric measure space of dimension $d$, and let $S$ be a measurable subset of
$X$ with finite measure. Then for any $0\leqslant t\leqslant\mu(S)$, there exists a subset $T\subseteq
S$ with $\mu(T)=t$.
\end{corollary}

\begin{proof}
The result is trivial for $t=0$ and for $t=\mu(S).$ Assume therefore $0<t<\mu(S)$ and
let $T_{0}=\emptyset$, $S_{0}=S,$ and $t_{0}=t.$ For every integer $n\geqslant1,$
let $k_{n}$ be the smallest integer such that for all cubes $Q_{\alpha}%
^{k_{n}}$ of generation $k_{n}$ we have $\mu\left(  S_{n-1}\cap Q_{\alpha
}^{k_{n}}\right)  \leqslant t_{n-1} .$ Let $T_{n}$ be the union
of sets of the form $S_{n-1}\cap Q_{\alpha}^{k_{n}} $, in such a way that its
measure is less than or equal to $t_{n-1}  ,$ but if we
attach one more such set, its measure becomes greater than $t_{n-1} .$ 
Let $t_{n}=t_{n-1}-\mu\left(  T_{n}\right)  ,$ and $S_{n}=S_{n-1}\setminus T_{n}$. Notice that all the sets $T_{n}$ are
pairwise disjoint, and $t_{n}  \leqslant c_{2}a_{1}^{d}%
\delta^{k_{n}d}.$ Now let
\[
T=\cup_{n=1}^{+\infty}T_{n}.
\]
Finally, $\mu\left(  T\right)  =\sum_{n=1}^{+\infty}\mu\left(  T_{n}\right)
=\sum_{n=1}^{+\infty}\left(  t_{n-1}-t_{n}\right) =\lim_{n\rightarrow+\infty}\left(  t_{0}-t_{n}\right) 
=t  .$
\end{proof}

\section{The construction of the equal measure partition}

\label{FS-section} In this section, we describe the construction of the equal
measure partition, and state its key properties. The construction is a
modified version of the construction given in the proof of \cite[Lemma
21]{FeiS02}. The main modification is the use of David and Christ's dyadic
cubes in place of Voronoi cells. The description of the construction is
contained in the proof of the main theorem of this paper, as follows.

\begin{theorem}
\label{Feige-Schechtman-lemma} Let $(X,\rho,\mu)$ be a connected Ahlfors
regular metric measure space of dimension $d$ and finite measure. Then there
exist positive constants $c_{3}$ and $c_{4}$ such that for every sufficiently large
$N$, there is a partition of $X$ into $N$ regions of measure $\mu\left(
X\right)  /N$, each contained in a ball of radius $c_{3}N^{-1/d}$ and
containing a ball of radius $c_{4}N^{-1/d}$.
\end{theorem}
 
Note that the requirement that $X$ be connected is essential, even if we give up
the condition on the small ball contained in each region. Indeed, observe first that 
$X$ is a totally bounded complete metric space, and is therefore compact. If $X$ is
not connected then it is the disjoint union of two 
open sets $Y$ and $W,$ both with positive measure by Ahlfors regularity. Clearly $Y$ and $W$ are also closed, and therefore compact. It follows that the
distance $\inf\{\rho(y,w):y\in Y,\,w\in W\}$ between them is positive, say
$\varepsilon.$ Let's now take $N$ large enough that $c_{3}N^{-1/d}%
<\varepsilon.$ If a partition as in the theorem exists for one such $N$, then
each region is contained either in $Y$ or in $W$. It follows that $\mu\left(
Y\right)  $ and $\mu\left(  W\right)  $ are integer multiples of $\mu\left(
X\right)  /N$, and this cannot be true for all such $N.$
 
\begin{proof}
Let $(X,\rho,\mu)$ be a connected Ahlfors regular metric measure space of
dimension $d$ with constants $0<c_{1}<c_{2}$ as per Definition
\ref{Ahlfors-def}. In addition, let
\[
\left\{Q_\alpha^k:k\in\mathbb Z,\,\alpha\in I_k\right\}%
\]
be a family of dyadic cubes on $X,$ as per Theorem \ref{DC-thm-dyadic-cubes},
with properties as per Definition \ref{DC-dyadic-cube-props-def}. For the sake
of simplicity, assume without loss of generality that $\mu\left(  X\right)
=1.$ We construct a partition of $X$ into $N$ regions of measure $1/N$ as follows.

\begin{enumerate}
\item \label{FS-step-max-gen} Assume%
\[
N\geqslant\frac{2}{c_{1}\delta^{d}\mathrm{diam}\left(  X\right)  ^{d}}%
\]
and let
 $n$ be the only integer such that
\begin{equation}
\label{characterization-n}a_{0}\delta^{n+1}<\left( \frac{2}{c_{1}%
N}\right) ^{1/d}\leqslant a_{0}\delta^{n}.
\end{equation}
In particular,%
\[
\mu(Q_{\alpha}^{n})\geqslant\mu(B(z_{\alpha}^{n},a_{0}\delta^{n}))\geqslant
c_{1}a_{0}^{d}\delta^{nd}\geqslant2/N
\]
for all $\alpha\in I_{n}$.

\item \label{FS-step-graph} Create a graph $\Gamma$ with a vertex
corresponding to each index $\alpha\in I_{n}$ and an edge $(\alpha,\beta)$
corresponding to each pair of centre points $z_{\alpha}^{n},\,z_{\beta}^{n}$
such that $B(z_{\alpha}^{n},a_1\delta^{n})\cap B(z_{\beta}^{n},a_1\delta^{n}%
)\neq\emptyset$. Observe that this graph is connected. Indeed, assume by
contradiction that it is not. Then $I_{n}$ can be decomposed into the disjoint
union of two sets $I_{n}^{1}$ and $I_{n}^{2},$ such that for any $\alpha\in
I_{n}^{1}$ and $\beta\in I_{n}^{2}$ we have $B(z_{\alpha}^{n},a_1\delta^{n})\cap
B(z_{\beta}^{n}a_1\delta^{n})=\emptyset.$ But then $X_{1}=\cup_{\alpha\in
I_{n}^{1}}B(z_{\alpha}^{n},a_1\delta^{n})$ and $X_{2}=\cup_{\beta\in I_{n}^{1}%
}B(z_{\beta}^{n},a_1\delta^{n})$ are two disjoint open sets that cover $X,$ and
this is a contradiction because $X$ is connected.

\item \label{FS-step-tree} Take any spanning tree $S$ of $\Gamma$
\cite[Section 6.2]{Ore62}. The tree $S$ has leaves, which are nodes having
only one edge, and either a single centre node, or a bicentre, which is a pair
of nodes joined by an edge. The centre or bicentre nodes are the nodes for
which the shortest path to any leaf has the maximum number of edges
\cite[Chapter 6, Section 9]{Rio58}. If there is a single centre, mark it as
the root node. If there is a bicentre, arbitrarily mark one of the two nodes
as the root node.

Create the directed tree $T$ from $S$ by directing the edges from the leaves
towards the root \cite[Chapter 6, Section 7]{Rio58}. In other words,
$(\alpha,\beta)\in T$ means that $T$ contains an edge directed from the child
node $\alpha$ towards the parent node $\beta$.

\item \label{FS-step-M} Set
\[
M:=\frac{2c_{2}\left(  3a_{1}\right)  ^{d}}{c_{1}(\delta a_{0})^{d}}.
\]
By Corollary \ref{DC-big-ball-cor}, for all $\beta\in I_{n},$%
\[
Q_{\beta}^{n}\cup\bigcup_{(\alpha,\beta)\in T}Q_{\alpha}^{n}\subset B\left(
z_{\beta}^{n},  3a_{1}  \delta^{n}\right)  ,
\]
so that%
\[
\mu\left(  Q_{\beta}^{n}\cup\bigcup_{(\alpha,\beta)\in T}Q_{\alpha}%
^{n}\right)  \leqslant\mu\left(  B\left(  z_{\beta}^{n},3 a_{1}
\delta^{n}\right)  \right)  \leqslant c_{2}\left(  3a_{1}\delta^n\right)  ^{d},
\]
but by \eqref{characterization-n}%
\begin{equation}
\label{delta-bound}
\delta^{nd}<\frac2{ c_1 \left(\delta a_{0}\right)^dN},
\end{equation}
and therefore
\begin{equation}
\label{en-beta}
\mu\left(  Q_{\beta}^{n}\cup\bigcup_{(\alpha,\beta)\in T}Q_{\alpha}%
^{n}\right)  <\frac{2c_{2}\left( 3 a_{1}\right)  ^{d}}{c_1\left(\delta a_{0}\right)^d N}=\frac{M}{N}.
\end{equation}

\item
\label{FS-step-m} Let $m:=n+k,$ where $k$ is a positive integer such that
$
\delta^{k}\leqslant 3M^{-2/d}.
$
By \eqref{delta-bound}, this ensures that for all $\eta\in I_{m}$
\begin{equation}
\mu(Q_{\eta}^{m})
\leqslant
\mu\left(  B\left(  z_{\eta}^{m},a_{1}\delta ^{m}\right)  \right)  
\leqslant 
c_{2}\left(  a_{1}\delta^{n+k}\right)^{d}
< 
\frac{2 c_{2}a_{1}^{d}\delta^{kd}}{c_1\left(\delta a_{0}\right)^dN}
\leqslant\frac{1}{MN}.
\label{small-cubes}%
\end{equation}

\item
\label{FS-step-leaf}For each leaf node $\beta$, define $N_{\beta}:=\lfloor
\mu(Q_{\beta}^n)N\rfloor.$ This is the maximum number of sets of measure
$1/N$ that fit inside $Q_{\beta}^{n}$. Choose $N_{\beta}$ cubes of generation
$m$ inside $Q_{\beta}^{n}$ to be the nuclei of each region of the partition.
This can be done because the bounds \eqref{en-beta} and \eqref{small-cubes}
ensure that the total measure of the $N_{\beta}$ cubes is at most $1/N$, while
$Q_{\beta}^{n}$ has measure at least $2/N$.\emph{ }By Corollary
\ref{DC-measurable-subset-corollary}, extend each nucleus into a region of measure
$1/N$ within $Q_{\beta}^{n}$. Call $W_{\beta}$ the remainder set inside
$Q_{\beta}^{n}$. Clearly $\mu\left(  W_{\beta}\right)  <1/N$.
For future reference, for leaf nodes $\beta$ we will set $X_\beta:=Q_\beta^n.$

\item \label{FS-step-nonleaf} For each nonleaf node $\beta$ other than the
root, define $X_{\beta}:=Q_{\beta}^{n}\cup\bigcup_{(\alpha,\beta)\in
T}W_{\alpha}$, that is, we add all the remainders of the children of $\beta$
in $T$ to $Q_{\beta}^{n}$ to obtain $X_{\beta}$. We can recursively assume
that if $\left(  \alpha,\beta\right)  \in T$, then $W_{\alpha}\subset
Q_{\alpha}^{n}.$ Since $X_{\beta}\subset Q_{\beta}^{n}\cup\bigcup
_{(\alpha,\beta)\in T}Q_{\alpha}^{n}$, then $\mu\left(  X_{\beta}\right)  \leqslant
M/N.$ Now define $N_{\beta}:=\lfloor\mu(X_{\beta})N\rfloor$ and once again
choose $N_{\beta}$ cubes of generation $m$ inside $Q_{\beta}^{n}$ to be the
nuclei of each region of the partition. This can be done because the bounds
\eqref{en-beta} and \eqref{small-cubes} ensure that the total measure of the
$N_{\beta}$ cubes is at most $1/N$, while $Q_{\beta}^{n}$ has measure at least
$2/N$. By Corollary \ref{DC-measurable-subset-corollary}, take a subset
$W_{\beta}$ of $Q_{\beta}^{n},$ disjoint from the previously chosen nuclei, of
measure $\mu\left(  X_{\beta}\right)  -N_{\beta}/N<1/N$.\emph{ }This is again
possible because $\mu(Q_{\beta}^{n})\geqslant2/N$ and the total measure of the
nuclei is at most $1/N$.\emph{ }Finally, once again by Corollary
\ref{DC-measurable-subset-corollary}, extend each nucleus into a region of measure
$1/N$ within $X_{\beta}\setminus W_{\beta}$.

\item \label{FS-step-root} For the root node $\gamma$, we define $X_{\gamma
}:=Q_{\gamma}^{n}\cup\,\bigcup_{(\beta,\gamma)\in T}W_{\beta}$. We have
$\mu(X_{\gamma})=N_{\gamma}/N$, where
\[
N_{\gamma}:=N-\sum_{\alpha\neq\gamma}N_{\alpha}%
\]
is an integer. Since $X_{\gamma}\subset Q_{\gamma}^{n}\cup\bigcup
_{(\beta,\gamma)\in T}Q_{\gamma}^{n}$, then $\mu\left(  X_{\gamma}\right)
\leqslant M/N.$ Once again choose $N_{\gamma}$ cubes of generation $m$ inside
$Q_{\gamma}^{n}$ to be the nuclei of each region of the partition. This can
be done because the bounds \eqref{en-beta} and \eqref{small-cubes} ensure that
the total measure of the $N_{\gamma}$ cubes is at most $1/N$, while $Q_{\gamma
}^{n}$ has measure at least $2/N$.\emph{ }Finally, once again by Corollary
\ref{DC-measurable-subset-corollary}, extend each nucleus into a region of measure
$1/N$ within $X_{\gamma}.$
\end{enumerate}

By the definition of the graph $\Gamma$ used to create the tree $T$, if
$(\alpha,\beta)\in T$ then the two balls $B(z_{\alpha}^{n},a_1\delta^{n})$ and
$B(z_{\beta}^{n},a_1\delta^{n})$ overlap. Therefore, by Corollary
\ref{DC-big-ball-cor}, the union of cubes $Q_{\beta}^{n}\cup\bigcup
_{(\alpha,\beta)\in T}Q_{\alpha}^{n}$ is contained in the ball $B(z_{\beta
}^{n}, 3a_{1}  \delta^{n})$. Since steps \ref{FS-step-leaf},
\ref{FS-step-nonleaf} and \ref{FS-step-root} of the construction have ensured
that
\[
X_{\beta}\subset Q_{\beta}^{n}\cup\bigcup_{(\alpha,\beta)\in T}Q_{\alpha
}^{n},
\]
and that each region is contained in
some $X_{\beta}$, each region is therefore contained in a ball of radius
$3a_{1} \delta^{n}$. But by \eqref{delta-bound}
\[
\delta^{n}<\frac 1{\delta a_0}\left(  \frac{2}{c_{1}N}\right)  ^{1/d},
\]
so that each region is contained in a ball of radius $c_{3}N^{-1/d}$, where
\[
c_{3}:=\frac{3a_1}{\delta a_0}  \left(  \frac{2}{c_{1}}\right)  ^{1/d}.
\]

Steps \ref{FS-step-leaf},
\ref{FS-step-nonleaf} and \ref{FS-step-root} of the construction also ensure that each region
contains a dyadic cube $Q_{\eta}^{m}$ of generation $m=n+k$. By property
\eqref{DC-prop-inner-ball} of Definition \ref{DC-dyadic-cube-props-def} this
implies that each region contains a ball $B(z_{\eta}^{m},a_{0} \delta^{m})$ of
radius $a_{0} \delta^{m}$. From \eqref{characterization-n}, we have
\begin{align*}
\delta^{m}  &  
= 
\delta^{n} \delta^{k} 
\geqslant
 \left(  \frac{2}{c_{1} a_{0}^{d} N}\right)  ^{1/d} {\delta}^k,
\end{align*}
so that each region contains a ball of radius $c_{4} N^{-1/d}$, where
\begin{align*}
c_{4}  &  :=  \left(  \frac{2}{c_{1}}\right)  ^{1/d}{\delta^k} .
\end{align*}
\end{proof}

As mentioned in the Introduction, in the original argument of Feige and Schechtman for the case of the sphere or a manifold, one can start with a saturated packing of balls of radius $r=(c_1 N)^{-1/d}$, and then use the Voronoi cells relative to the centres of the balls as ``pre-regions'' to be split into regions of measure $1/N$ plus a remainder. Unfortunately, in the case of an Ahlfors regular space, two Voronoi cells
may overlap on a set of non-zero measure and the argument cannot be applied. 

Consider for example the set $X:=\{(x,y)\in\mathbb R^2:xy=0, |x|\leqslant1,|y|\leqslant1\}$ with distance $\rho((x_1,y_1),(x_2,y_2)):=\max(|x_1-x_2|,|y_1-y_2|),$ and the $1$-dimensional Hausdorff measure. This is an Ahlfors regular space of dimension $1$. The Voronoi cells 
relative to the two points $(1,0)$ and $(-1,0)$ are $V_1:=\{(x,y)\in\mathbb R^2:xy=0, 0\leqslant x\leqslant 1,|y|\leqslant1\}$ and  $V_{-1}:=\{(x,y)\in\mathbb R^2:xy=0, -1\leqslant x\leqslant 0,|y|\leqslant1\}$,
and the $1$-dimensional Hausdorff measure of their intersection is $2$.

\section{Applications to numerical integration}\label{numerical}

The above Corollary \ref{quasi-equal} and Theorem \ref{Feige-Schechtman-lemma}
have interesting applications to numerical integration { on manifolds or more general metric spaces}. For example, in \cite[Corollary 2.15]{BCCGST}
the following result is proven. Let $X$ be a smooth compact $d$-dimensional Riemannian manifold without boundary, with Riemannian distance $\rho$ and Riemannian measure $\mu$. For any $1\leqslant p\leqslant+\infty,$  $\alpha>d/p$ and $\kappa\geqslant 1/2$,
there exists a constant $c$ such that if $\{z_j\}_{j=1}^N$ is a distribution
of points on $X$ with
\begin{equation*}
\label{mesh-separation-ratio}
\frac{\sup_{x\in X}\min_{j}\rho(x,z_j)}{\min_{i\neq j}\rho(z_i,z_j)}\leqslant \kappa
\end{equation*}
then there exist positive weights $\{\omega_j\}_{j=1}^N$ such that
\[
\left|\sum_{j=1}^N\omega_jf(z_j)-\int_Xf(x)d\mu(x) \right|\leqslant c N^{-\alpha/d}\left\|f\right\|_{W^{\alpha,p}}
\]
for all functions $f\in W^{\alpha,p}.$ Here $W^{\alpha,p}$ denotes the Sobolev class of functions $f$ with $(I+\Delta)^{\alpha/2}f\in L^p(X).$
If we take the nodes $\{z_j\}_{j=1}^N$ to be exactly the centres of the inner balls of each of the $N$ regions of Corollary \ref{quasi-equal} or Theorem \ref{Feige-Schechtman-lemma}, then clearly 
\begin{equation*}
\frac{\sup_{x\in X}\min_{j}\rho(x,z_j)}{\min_{i\neq j}\rho(z_i,z_j)}\leqslant
\frac{2c_3 N^{-1/d}}{2 c_4 N^{-1/d}}=\frac{c_3}{c_4},
\end{equation*}
and we obtain a quadrature rule for every $N,$ with optimal decay exponent $-\alpha/d$ of the error on Sobolev classes (see \cite[Corollary 2.15 and Theorem 2.16]{BCCGST} for the details).

{ After the observations of Mandelbrot that nature can often be represented more faithfully
by structures more general than manifolds, in the last decades there has been an increasing interest in analysis on fractals (see R. Strichartz's survey \cite{Str99} and the references therein, especially J. Kigami's book \cite{Kig01}).

In this context, numerical integration also plays a useful role, for example, in the search of approximate solutions to fractal differential equations 
by means of analogs of the finite-element method (see again \cite[Page 1204]{Str99}). Observe that in \cite{SU00} Strichartz and Usher 
develop efficient numerical integration methods on certain types of fractals, analogous to Simpson's method.

An application of Corollary \ref{quasi-equal} and Theorem \ref{Feige-Schechtman-lemma} to general metric measure spaces
can be found in \cite{BCCGT}.} 
There the authors study the error of numerical integration on metric measure spaces adapted to a partition of the space into $N$ disjoint subsets having diameter approximately $N^{-1/d}$ and measure approximately $N^{-1}$.  Without going into details, integrals are approximated with Riemann sums with nodes randomly taken from each region of the partition, and weights 
equal to the measure of the subsets.  Once again, Corollary \ref{quasi-equal} and Theorem \ref{Feige-Schechtman-lemma} guarantee the existence of one such decomposition under appropriate hypotheses
on the metric measure space.
{ For example, the Sierpinski gasket is a connected Ahlfors regular metric measure space of dimension $d=\log 3/\log 2$ (see \cite{DavS97}), and Theorem \ref{Feige-Schechtman-lemma} applies}.
In particular, the decomposition obtained with Theorem \ref{Feige-Schechtman-lemma} provides quadrature rules with weights equal to $1/N.$

\section{Final remarks}
{
The above results can be slightly generalized to the case where the measure of the balls is controlled above and below by some
function $\Phi(r)$ of the radius that satisfies a doubling condition, as in \cite{Rez15}. More precisely,
\begin{theorem}
\label{CD-extension}
Let $(X,\rho,\mu)$ be a connected metric measure space with finite measure. 
Assume that  $\Phi$ is a continuous non-negative strictly increasing function on $(0,+\infty)$
such that $\Phi(r)\to 0$ as $r\to0^+$ and satisfying the strict doubling property,
i.e. for some constants $b_1,b_2>1$ and  $r_0>0$ and for any $r\leqslant r_0$ it holds that
\[
b_1\Phi(r)\leqslant\Phi(2r)\leqslant b_2\Phi(r).
\]
Assume also that there exist positive constants $c_1,c_2$ such that 
for all $x\in X$ and for all $r\leqslant r_0$,
\[
c_1\Phi(r)\leqslant\mu(B(x,r))\leqslant c_2\Phi(r).
\]
Then there
exist positive constants $c_{3}$ and $c_{4}$ such that for every sufficiently large
$N$, there is a partition of $X$ into $N$ regions of measure $\mu\left(
X\right)  /N$, each contained in a ball of radius $c_{3}\Phi^{-1}(\mu(X)/N)$ and
containing a ball of radius $c_{4}\Phi^{-1}(\mu(X)/N)$.
\end{theorem}
We sketch the proof of this result. Observe first that 
Christ's dyadic decomposition holds under the weaker hypotheses that $\rho$ { is} a metric
(actually, it suffices that $\rho$ be a quasi-metric such that the balls $\{x:\rho(x,y)<r\}$ are open)
and that $\mu$ { is} a doubling measure, i.e. $\mu(B(x,2r))\leqslant c\mu(B(x,r)).$ 
This immediately allows to extend Corollaries \ref{DC-big-ball-cor} and \ref{DC-measurable-subset-corollary}
to the new general context described in Theorem \ref{CD-extension}. The construction of the partition follows the lines of the Ahlfors regular case,
with a few slight modifications that we leave to the reader.
}

{\bf Acknowledgements.}
{ The authors wish to thank Alexander Reznikov for suggesting to study the extension now contained in Theorem \ref{CD-extension}.}
Thanks to the Universit\`a degli studi di Bergamo for hosting the second author for a one month visit,
funded by the project ``Progetto ITALY\textregistered - Azione 3: Grants for Visiting Professor and Scholar''.
Part of the work was conducted while the second author was attending the  
Program on ``Minimal Energy Point Sets, Lattices, and Designs'' at the Erwin Schr\"odinger International Institute
for Mathematical Physics, 2014, as a Visiting Fellow of the Australian National University.
The following Pozible supporters made this attendance possible:
S.v.A., A.D., C.F., O.F., K.M., J.P., A.H.R., R.R., E.S., M.T.;
Yvonne Barrett, Angela M. Fearon, Sally Greenaway, Dennis Pritchard, Susan Shaw, Bronny Wright;
Naomi Cole;
the Russell family, Jonno Zilber;
Jennifer Lanspeary, Vikram.


\end{document}